\documentclass{amsart}
\usepackage{amsfonts}
\usepackage{graphicx}
\usepackage{amscd}
\usepackage{amsmath}
\usepackage{amssymb}
\theoremstyle{plain}

\newtheorem{thm}{Theorem}
\numberwithin{equation}{section}

\newtheorem{dfn}{Definition}
\newtheorem{rmk}{Remark}
\newtheorem{thm*}{Theorem}
\newtheorem{cor}{Corollary}
\newtheorem{prop}{Proposition}

\newtheorem{ex}{Example}

\newcommand{\dif}{{\rm d}}

\newcommand{\fof}{f\!\otimes\!f}

\DeclareMathOperator\sgn{sgn}

\begin{document}
\title[First variation]{First variation of the Hausdorff measure of non-horizontal submanifolds in sub-Riemannian stratified Lie groups}
\author{Marcos M. Diniz$^{1}$}
\address{$^{1}$ Faculdade de Matem\'{a}tica-UFPA,\\
66075-110-Bel\'em-PA, BR}
\email{mdiniz@ufpa.br}
\thanks{$^{1}$ Partially supported by CAPES-BR}
\author{Maria R. B. Santos$^{2}$}
\address{$^{2}$Departamento de Matem\'atica-UFAM,\\
 69077-070-Manaus-AM-BR}
\email{mrosilenesantos@gmail.com}
\thanks{$^{2}$ Partially supported by CAPES-BR}
\author{Jos\'{e} M. M. Veloso$^{3}$}
\address{$^{3}$ Faculdade de Matem\'{a}tica-UFPA,\\
66075-110-Bel\'em-PA, BR}
\email{veloso@ufpa.br}
\thanks{$^{3}$ Partially supported by CAPES-BR}
\begin{abstract}
We determine necessary conditions for a non-horizontal submanifold of a sub-Riemannian stratified Lie group to be of minimal measure. We calculate the first variation of the measure for a non-horizontal submanifold and find that the minimality condition implies the tensor equation $H+\sigma=0$, where $H$ is analogous to the mean curvature and $\sigma$ is the \emph{mean torsion}. We also discuss new examples  of minimal non-horizontal submanifolds in the Heisenberg group, in particular surfaces in $\mathbb H^2$.
\end{abstract}
\maketitle








\section{Introduction}\label{intro}
The last years have seen a generalization of Riemannian geometry to sub-Rie\-ma\-nnian geometry \cite{agrachev2012, gromov144, montgomery2006, strichartz1989, strichartz1986}.
A  sub-Riemannian manifold is a connected manifold $G$ with a distribution $D\subset TG$ such that successive Lie brackets fields in $D$ generate all the tangent space $TG$. In addition, a positive definite scalar product $\langle\cdot,\cdot\rangle$ is defined at $D$ such that it is possible to calculate the length of admissible curves, i.e., the tangent curves on $D$. As in Riemannian geometry, the distance $\rho$ between two points $p$ and $q$ in $G$ is defined as the infimum length of admissible curves that connect the points $p$ and $q$. By means of this distance, the sub-Riemannian manifold becomes a metric space \cite{Bellaiche1996} and can be endowed with a Hausdorff measure. Also, there is an equivalent of the Riemannian volume, the so called Popp's volume. Introduced in \cite{montgomery2006}, it is a smooth volume which is canonically associated with the sub-Riemannian structure and was used in \cite{agrachev2012} to define the sub-Laplacian in sub-Riemannian geometry.

On sub-Riemannian manifolds which are nilpotent Lie groups (stratified Lie gro\-u\-ps) the spherical Hausdorff measure, the Popp's measure and the Haar measure are mutually constant multiples one of another \cite{Agrachev2012a, ghezzi2014, Magnani2005, mitchell1985, barilari2013}. Then, it is natural to ask what means a submanifold of
a stratified Lie group to be of minimal measure.

The aim of this paper is, using the measure proposed by Magnani and Vittone for non-horizontal submanifolds in stratified Lie groups \cite{Magnani2008}, to calculate the first variation of these submanifolds and determine the necessary conditions for a non-horizontal submanifold to be of minimal measure. We will also give new examples of non-horizontal minimal submanifolds in the Heisenberg group, in particular minimal surfaces in 5-dimensional Heisenberg group . In order to present the main results, we briefly introduce some concepts which will be dealt with in detail in the subsequent sections.

Let $\mathbb G$ be a stratified Lie group with a graded  Lie algebra $\mathfrak{g}=\mathfrak{g}^1\oplus\cdots\oplus\mathfrak{g}^r$ with $r\geq1$. If $\langle\cdot,\cdot\rangle$ is a scalar product on $\mathfrak g^1$, we can extend it to $\mathfrak g$ by induction (see Proposition \ref{extend}).
We also consider the distribution $D\subset T\mathbb G$ generated by $\mathfrak g^1$ and the scalar product in $D$ generated by $\langle\cdot,\cdot\rangle$. This way, $(\mathbb G,D,\langle\cdot,\cdot\rangle)$ becomes a sub-Riemannian manifold, also called the Carnot group. Note that, traditionally the scalar product on $D$ is extended to $T\mathbb G$ and the Riemannian connection on $T\mathbb G$ is used to make calculations on $\mathbb G$. This observation will be our starting point. We shall work with a covariant derivative $\overline\nabla$ defined by $\overline\nabla X=0$ for all $X\in\mathfrak g$. It has intrinsic torsion which is essentially the negative of the Lie bracket in $\mathbb G$ and the zero curvature tensor. So, this covariant derivative permits us to establish an interesting parallel between the invariants of submanifolds in $\mathbb R^n$ and the invariants of submanifolds in $\mathbb G$.

The geometry of a submanifold $M$ of $\mathbb G$ at each point depends on the relative position of
$TM$ and $D$. The submanifolds with a ``high'' contact with $D$ at one point may have
singularities from the metric point of view, even though being a $C^\infty$ submanifold. In this paper,
we avoid these situations by considering a non-horizontal submanifold $M$ transverse to D,
i.e., $TM+D=T\mathbb G$.  For this submanifold the horizontal normal subspace $TM^\perp$ play the same role as the normal space does to a submanifold in $\mathbb R^n$. We define $TM^\perp$ as the orthogonal subspace to $TM\cap D$ in $D$, i.e., the horizontal distribution $D=(TM\cap D)\oplus TM^\perp$. Hence, $T\mathbb G=TM\oplus TM^\perp$ and we can use this decomposition
(in general not orthogonal!) to project $\overline\nabla$ to a connection $\nabla$ over $TM$.

Let $e_1,\ldots, e_n$ be a orthonormal basis of $\mathfrak{g}$ with its dual $e^1,\ldots,e^n$. Let $f_1,\ldots,f_n$ be a adapted frame in the $T\mathbb{G}$ such that $f_1,\ldots,f_p$ are orthogonal to $TM\cap D$ and $f_{p+1},\ldots,f_{d_1}$ is a orthonormal basis of $TM\cap D$ in $D$. We complete $f_{p+1},\ldots,f_{d_1}$ to a basis $f_{p+1},\ldots,f_n$ of $TM$ taking
$
f_j=e_j-\sum_{\alpha=1}^pA_j^\alpha f_{\alpha},
$
for $j=d_1+1,\ldots,n$. If we denote by $f^1,\ldots, f^n$ its dual basis, then the sub-Riemannian volume form on $\mathbb G$ is defined as $\dif V=e^1\wedge\cdots\wedge e^n=f^1\wedge\cdots\wedge f^n$. When $M$ is a hypersurface, the $H$-perimeter measure is traditionally used \cite{Montefalcone2007, Montefalcone2012, Danielli2007, Hladky2013}. For the non-horizontal submanifolds of codimension $p\geq 1$, the spherical Hausdorff measure has the following representation proved in \cite{Magnani2008}:
\begin{equation}\label{M}
\int_M\theta(\tau^d_M(x))\dif S^d_\rho(x)=\int_M|\tau^d_{M}(x)|\dif\mbox{vol}_{h}(x) \ ,
\end{equation}
where $h$ is a fixed Riemannian metric, $d$ is the Hausdorff dimension of $M$, $\theta(\tau^d_M(x))$ is the metric factor (see Section \ref{hspace}), $S^d_\rho$ is the	 $d$-spherical Hausdorff measure and $\dif\mbox{vol}_{h}$ is the Riemannian volume form on $M$ induced by $(\mathbb G,h)$.

We will denote by $\mu$ the measure in non-horizontal submanifolds defined by  $\dif\mu(x)=|\tau^d_{M}(x)|\dif\mbox{vol}_{h}(x)$, which  is a nat\-u\-ral can\-di\-date to de\-fine the vol\-ume for non-horiz\-ontal submanifolds.
If the metric factor $\theta(\tau^d_M(x))$ is constant on $M$ (the case of the Heisenberg group $\mathbb{H}^n$, see Section \ref{hspace}) the measure $\dif\mu(x)$ is a multiple of $(Q-p)$-spherical Hausdorff measure on $M$. Writing the density $\dif\mu$ in the  adapted frame  $f^1,\ldots,f^n$ we obtain a  beautiful formula which we will use for the variational calculation, namely
$
\dif \mu=f^{p+1}\wedge\cdots\wedge f^n
$ (see Theorem \ref{dS}).

The second result of this paper (Theorem \ref{firstvar}) is a sufficient condition for minimality of non-horizontal submanifolds. We say that a non-horizontal submanifold is minimal if $H+\sigma=0$ on $TM^\perp$, where $H$ is the mean curvature (Definition \ref{curvmedia}) and $\sigma$ is the mean torsion (Definition \ref{colchete}). In the case of hypersurfaces, the mean torsion is null and so the definition of minimality is the same as in \cite{Montefalcone2007, Montefalcone2012, Danielli2007, Hladky2013, Hurtado2010, Ritore2008} and also of minimal submanifolds of Riemannian geometry.

As a direct application of Theorem \ref{firstvar} we present the following example: if $M$ is minimal submanifold of $\mathbb R^{2n}$, then $N=M\times\mathbb R$ is minimal submanifold of $\mathbb{H}^n$.
Furthermore, an interesting application of this theorem is discussed in section \ref{H2}, where we find minimal non-horizontal surfaces in the $5$-dimensional Heisenberg group $\mathbb H^2$. Observe that in the $3$-dimensional Heisenberg group $\mathbb H^1$ the tangent horizontal curves to minimal surfaces are lines and hence this surfaces are ruled surfaces \cite{Pansu1982}. Now, for $\mathbb H^2$ the tangent horizontal curves to minimal surfaces can be more general. In section \ref{H2}, we will present two cases: the curves are lines and we obtain ruled surfaces; the curves are circles and we obtain a family of circles, which we will call tubular surfaces.

In the last section, we prove for non-horizontal hypersurfaces the following: $-H_{f_1}=\mbox{div}_{\mathbb G}\Big(\frac{\mbox{grad}_{D}\phi}{|\mbox{grad}_{D}\phi|}\Big)$, where $\phi:\mathbb G\rightarrow\mathbb R$ is smooth function, $\mbox{div}_{\mathbb G}$ is the divergence function on $\mathbb G$ and $\mbox{grad}_D$ is the horizontal gradient operator. With this formula we give a proof that the hyperboloid paraboloid is minimal.

\section{Stratified Lie groups}\label{slg}

A stratified Lie group $\mathbb{G}$ is an $n$-dimensional connected, simply connected nilpotent Lie group whose Lie algebra $\mathfrak{g}$ decomposes as $\mathfrak{g}=\mathfrak{g}^1\oplus\mathfrak{g}^2\oplus\cdots\oplus\mathfrak{g}^r$ and satisfies the condition $[\mathfrak{g}^1,\mathfrak{g}^j]=\mathfrak{g}^{j+1}, j=1,\ldots, r-1,\quad [\mathfrak{g}^j,\mathfrak g^r]=0, j=1,\ldots,r$. Write $d_i=\dim\mathfrak{g}^i$, choose a basis $e_1,\ldots,e_n$ of $\mathfrak{g}$ such that $e_{d_{j-1}+1},\ldots,e_{d_j}$ is a basis of $\mathfrak{g}^j$ and denote its dual basis by $e^1,\ldots,e^n$. Then, we define the \emph{degree} of $e_k$ as $\deg k=j$ if $d_{j-1}<k\leq d_j$.

We also define an covariant derivative $\overline\nabla$ on $T\mathbb G$ by
\begin{equation*}
\overline\nabla e_j=0
\end{equation*}
for every $j=1,\ldots,n$. Clearly, the curvature $\overline K$ of $\mathbb G$ is null. If $\overline T$ is the torsion of $\overline\nabla$, then
$
\overline T(e_i,e_j)=-[e_i,e_j]=-\sum_{k=1}^nc_{ij}^ke_k$
or
\begin{equation*}
\overline T=-\frac 1 2\sum_{i,j,k=1}^nc_{ij}^ke^i\wedge e^j\otimes e_k,
\end{equation*}
where $c_{ij}^k$ are the structure constants of $\mathfrak{g}$ associated with $e_1,\ldots,e_n$.

The stratification hypothesis on $\mathfrak{g}$ implies that $c_{ij}^k=0$ if $\deg{k}\neq \deg{i}+\deg{j}$. In particular, $c_{ij}^k=0$ for $k=1,\ldots,d_1$. We can then write
\begin{equation*}
\overline T=\sum_{k=d_1+1}^n\overline T^k\otimes e_k,\,\,\, \mbox{where}\,\,\, \overline T^k=-\frac 1 2\sum_{i,j=1}^nc^k_{ij}e^i\wedge e^j\,.
\end{equation*}

Let $D\subset T\mathbb G$  be defined by $D_x=\{X(x)\,|\, X\in\mathfrak{g}^1\}$. Consider a positive definite scalar product $\left\langle \,,\right\rangle$ on  $D$ such that $e_1,...,e_{d_1}$ are orthonormal. With this scalar product, $(\mathbb G,D,\left\langle ,\right\rangle)$ becomes a \emph{sub-Riemannian manifold}.

\begin{prop}\label{extend}
There is a canonical extension of $\left\langle ,\right\rangle$ on $D$ to a scalar product  on $T\mathbb{G}$.
\end{prop}
\begin{proof}
The proof goes by induction. Suppose that we have extended $\left\langle , \right\rangle$ to a scalar product of $\mathfrak{g}^k$. Then we have a scalar product on $\mathfrak{g}^1\oplus \cdots \oplus\mathfrak{g}^k$, and we can consider the map
$B:\mathfrak{g}^1\otimes\mathfrak{g}^k\rightarrow \mathfrak{g}^{k+1}$ defined by $B(X\otimes Y)=[X,Y]$.
The scalar products on both $\mathfrak{g}^1$ and $\mathfrak{g}^k$ induce a scalar product on $\mathfrak{g}^1\otimes\mathfrak{g}^k$. Thus, $B:(\ker B)^\perp\rightarrow\mathfrak{g}^{k+1}$ is an isomorphism, by which we can transport the scalar product of $(\ker B)^\perp$ to $\mathfrak{g}^{k+1}$.
\end{proof}
We shall use the same symbol $\langle,\rangle$ for the extended scalar product. Therefore, from the definition of $\overline\nabla$ and Proposition \ref{extend}, we get for every $X,Y\in T\mathbb G$ that
\begin{equation*}
\overline\nabla\left\langle X,Y\right\rangle=\left\langle \overline\nabla X,Y\right\rangle+\left\langle X,\overline\nabla Y\right\rangle.
\end{equation*}
\begin{prop}
Suppose that $f:\mathfrak{g}\rightarrow\mathfrak{g}$ is a graded automorphism of Lie algebras such that
$\left.f\right|_{\mathfrak{g}^1}:\mathfrak{g}^1\rightarrow\mathfrak{g}^1$
is an isometry. Then $f$ is an isometry of $\mathfrak{g}$.
\end{prop}
\begin{proof} The proof goes by induction. We will suppose that we have shown that  $f:\mathfrak{g}^1\oplus\cdots\oplus\mathfrak{g}^k\rightarrow\mathfrak{g}^1\oplus\cdots\oplus\mathfrak{g}^k$ is an isometry. Consider $\fof:\mathfrak{g}^1\otimes\mathfrak{g}^k\rightarrow\mathfrak{g}^1\otimes\mathfrak{g}^k$ defined by $\fof(x\otimes Y)=f(x)\otimes f(Y)$. Then, $\fof$ is an isometry, i.e.,
\begin{align*}
\langle\fof(x\otimes Y),f\otimes f(x'\otimes Y')\rangle=&
\langle fx,fx'\rangle\langle  fY, fY'\rangle=\langle x,x'\rangle\langle  Y, Y'\rangle\\
=&\langle x\otimes Y,x'\otimes Y'\rangle.
\end{align*}
As before, let $B:\mathfrak{g}^1\otimes\mathfrak{g}^k\rightarrow\mathfrak{g}^{k+1}$ be defined by $B(x\otimes Y)=[x,Y]$. Then
\begin{equation*}
B(\fof(x\otimes Y))=B(f(x)\otimes f(Y))=[fx,fY]=f[x,Y]=f(B(x\otimes Y)).
\end{equation*}
Therefore, $\fof(\ker B)=\ker B$ and $\fof((\ker B)^\perp)=(\ker B)^\perp$.

Take $Z,Z'\in\mathfrak{g}^{k+1}$. Then,
$Z=B(u),\, Z'=B(u')$,
where $u,u'\in (\ker B)^\perp$. Note that $\fof (u),f\!\otimes\!f(u')\in (\ker B)^\perp$  and $B$, restricted to this subspace, is an isometry. So, we have
\begin{align*}
\langle fZ,fZ'\rangle=&\langle B(\fof (u)),B(\fof(u'))=\langle \fof (u),f\!\otimes\!f(u')\rangle\rangle=\langle u,u'\rangle\\
=&\langle B(u),B(u')\rangle=\langle Z,Z'\rangle\,.
\end{align*}
\end{proof}

\section{Non-horizontal submanifolds}\label{nhs}

Consider a stratified Lie group $\mathbb G$. A submanifold $M\subset\mathbb G$ of codimension $p$ is \emph{non-horizontal} if $TM+D=T\mathbb G$. We then have that $TM\cap D$ is a subvector bundle with $\dim TM\cap D=d_1-p$.

Let $U$ be an open neighborhood of $\mathbb G$ such that $M\cap U\neq \emptyset$. An adapted basis to $M$ on $U$ is a basis $f_1,\ldots,f_n$ on $U$ such that
\begin{equation*}
f_j=\sum_{k=1}^{d_1}a_j^ke_k,\,  j=1,\cdots,d_1,
\end{equation*}
is an orthonormal basis of $D$  with
$f_1,\ldots,f_p$  orthogonal to $TM\cap D$ on $D$ and $f_{p+1},\ldots,f_{d_1}$ is a basis of $TM\cap D$. It follows that the matrix $(a_j^k)_{1\leq j,k\leq d_1}$ is orthogonal.
We complete $f_{p+1},\ldots,f_{d_1}$ to a basis $f_{p+1},\ldots, f_n$ of $U$ by
\begin{equation}\label{fjj}
f_j=e_j-\sum_{\alpha=1}^pA_j^\alpha f_{\alpha},\, j=d_1+1,\ldots,n.
\end{equation}
such that $f_j$ is tangent to $M$ on $U$.

Now, let $i,j=1,\ldots,n$ and let $\overline\omega_j^i$ be the connection $1$-forms defined in terms of the basis $f_i$. Then
\begin{equation}\label{conex1}
\overline\nabla f_j=\sum_{i=1}^{d_1}\overline\omega _j^i\otimes f_i,
\end{equation}
where
\begin{align*}
\overline\omega _j^i&=\sum_{k=1}^{d_1}a^k_i\dif a_j^k, i,j=1,\ldots,d_1,\\
\overline\omega_j^\alpha&=-\dif A_j^{\alpha}-\sum_{\beta=1}^pA_j^{\beta}\overline\omega _\beta^\alpha, \alpha=1,\cdots,p, j=d_1+1,\ldots,n,\\
\overline\omega_j^i&=-\sum_{\beta=1}^pA_j^{\beta}\overline\omega _\beta^i,i=p+1, \ldots,d_1, j=d_1+1,\ldots,n.
\end{align*}

\begin{prop}\emph{({\bf{Cartan structural equations}})}\label{stru}
The $1$-forms $f^i$ and $\overline\omega_j^i$ satisfy
\begin{align*}
\dif f^\alpha=&-\sum_{j=1}^{n}\overline\omega^\alpha_j\wedge f^j+\tilde T^\alpha, \ \alpha=1,\ldots,p,\\
\dif f^{i}=&-\sum_{j=1}^n\overline\omega^i_j\wedge f^j, \ i=p+1,\ldots,d_1,\\
\dif f^{j}=&\,\,\overline T^j, \ j=d_1+1,\ldots,n,\\
\dif\overline\omega^k_i=&-\sum_{j=1}^{d_1}\overline\omega^k_j\wedge\overline\omega^j_i,\,
k=1,\ldots,d_1,  \ i=1,\ldots,n,
\end{align*}
where
$\widetilde T^\alpha =\sum_{k=d_1+1}^nA_k^\alpha\overline T^k , \alpha=1,\ldots,p$.

\end{prop}
\begin{proof}
These identities follow directly from the definition of the adapted basis to $M$.
\end{proof}


\subsection{The second fundamental form and the Weingarten operator}\label{sff}
We denote by $TM^\perp$ the subvector bundle of $T\mathbb G$ on $M$ generated by $f_1, \ldots,\!f_p$. Then $TM^\perp$  is orthogonal to $TM\cap D$ on $D$ and  $T\mathbb G=TM^\perp\oplus TM$ (in general non-orthogonal!). If $v\in T_p\mathbb G$, $p\in M$, we  write $v=v^\top + v^\perp$, where $ v^\top\in TM$ and $ v^\perp\in TM^\perp$. Therefore
\begin{equation*}
\overline\nabla_{X}Y=(\overline\nabla_{X}Y)^\top+(\overline\nabla_{X}Y)^\perp\,\,\, \mbox{and}\,\,\,\overline T(X,Y)=\overline T^\top(X,Y)+\overline T^\perp(X,Y)
\end{equation*}
for every $X,Y\in TM$.

We define the covariant derivative on $TM$ by
$
\nabla_XY=(\overline\nabla_XY)^\top,\, \forall\, X,Y\in TM
$.
\begin{dfn}\label{segforma}
Let $S:TM\times TM\rightarrow TM^\perp$ be the bilinear form defined by
$$S(X, Y)=(\overline\nabla_XY)^\perp.$$ We say that $S$
is the \emph{second fundamental form} associated to $M$.
\end{dfn}
Moreover, given $\xi\in TM^\perp$ and $X\in TM$ we write $
\overline\nabla_X\xi=-A_\xi(X)+\nabla^\perp_X\xi,
$
where
\begin{equation*}
A_\xi(X)=-(\overline\nabla_X\xi)^\top\in TM\,\,\, \mbox{and}\,\,\,\nabla^\perp_X\xi=(\overline\nabla_X\xi)^\perp\in  TM^\perp.
\end{equation*}
\begin{dfn} Let
$A:TM\times TM^ \perp \rightarrow TM$ be the bilinear form defined by $$A_\xi(X)=-(\overline\nabla_X\xi)^\top.$$ We say that $A$ is the Weingarten operator associated to $M$.
\end{dfn}
Observe that $A_\xi:TM\rightarrow TM\cap D$, since $D$ is invariant under $\overline\nabla$. Also, we have that $\nabla^\perp:\xi\in TM^\perp\rightarrow \nabla^\perp\xi\in (TM)^\ast\otimes TM^\perp$ is a linear connection on $TM^\perp$, where $(TM)^\ast$ is the dual space of $TM$.

\begin{rmk}
The second fundamental form is generally not symmetric, i.e., for every $X,Y\in TM$,
$
S(X,Y)-S(Y,X)=(\overline{\nabla}_XY-\overline{\nabla}_YX-[X,Y])^\perp=\overline{T}^\perp(X,Y).
$
\end{rmk}
Let us introduce
$
P:TM\times TM^\perp\rightarrow\mathbb{R}
$
 as $P(X,\xi)=\langle X,\xi\rangle$ and
\begin{align*}
\widetilde\nabla_X P(Y,\xi)=&X(P(Y,\xi))-P(\nabla_XY,\xi)-P(Y,\nabla^\perp_X\xi)\\
=&X\langle Y,\xi \rangle-\left\langle\nabla_XY,\xi\right\rangle-\langle Y,\nabla^\perp_X\xi\rangle.
\end{align*}
\begin{prop}For every $X,Y\in TM$ and $\xi\in TM^\perp$, we have that
\begin{equation*}
\left\langle S(X,Y),\xi\right\rangle=\widetilde\nabla_X P(Y,\xi)+\left\langle A_\xi(X),Y\right\rangle.
\end{equation*}
In particular, if $Y\in TM\cap D$, then $\left\langle S(X,Y),\xi\right\rangle=\left\langle A_\xi(X),Y\right\rangle$.
\end{prop}
\begin{proof}
The proof follows by Definition \ref{segforma} and by compatibility of $\overline\nabla$ with the extended scalar product $\left\langle,\right\rangle$. Clearly, $\widetilde\nabla_X P(Y,\xi)=0, \forall\,\, Y\in TM\cap D$.
\end{proof}
Let $K^\perp$ be the curvature of $\nabla^\perp$, i.e.,
$K^\perp(X,Y)\xi=\!\nabla^\perp_X\nabla^\perp_Y\xi-\nabla^\perp_Y\nabla_X\xi-\nabla^\perp_{[X,Y]}\xi$
and let us define
\begin{equation*}
\widetilde\nabla_X S(Y,Z)=\nabla^\perp_X(S(Y,Z))-S(\nabla_XY,Z)-S(Y,\nabla_XZ).
\end{equation*}
We then obtain a sub-Riemannian version of the fundamental equations.

\begin{thm}The curvature $K$ of $\nabla$, the curvature $K^\perp$ of $\nabla^\perp$, the second fundamental form $S$ and the Weingarten operator $A$ satisfy
\begin{enumerate}
\item  (Gauss equation) $K(X,Y)Z=A_{S(Y,Z)}(X)-A_{S(X,Z)}(Y)$;
\item (Codazzi equation) $\widetilde\nabla_YS(X,Z)-\widetilde\nabla_XS(Y,Z)-S(T(X,Y),Z)=0$;
\item (Ricci equation) $K^\perp(X,Y)\xi=S(X,A_\xi Y)-S(Y,A_\xi X)$,
\end{enumerate}
where $T$ is the torsion of $\nabla$.
\end{thm}
In coordinates, we have the following. Taking the tangential component and the horizontal orthogonal component of (\ref{conex1}) we obtain, respectively
\begin{equation}\label{A}
S(X,Y)=\sum_{\alpha=1}^p\sum_{j=p+1}^nf^j(Y)\overline\omega^\alpha_j(X)f_\alpha\,\,\, \mbox{and}\,\,\,  A_{f_\alpha}(X)=-\sum_{i=p+1}^{d_1}\overline\omega_\alpha^i(X)f_i.
\end{equation}
Furthermore, $\nabla^\perp_{X} f_\alpha=\sum_{\beta=1}^p\overline\omega_\alpha^\beta(X)f_\beta$.

We end this section with two definitions, the first of which being well known.
\begin{dfn}\label{curvmedia}
The \emph{mean curvature} of $M$ is the tensor $H:TM^\perp\rightarrow\mathbb R$ defined by
$H_\xi=-\mbox{\emph{trace} } A_\xi$.
\end{dfn}
\begin{dfn}\label{colchete}
 The \emph{mean torsion} of $M$ is the tensor $\sigma:TM^\perp\rightarrow\mathbb R$ defined by
\begin{equation*}
\sigma_{\xi}=\sum_{j=d_1+1}^{n}f^j(\overline T( \xi,f_j)).
\end{equation*}
\end{dfn}

\section{The $\mu$-measure of non-horizontal submanifolds}\label{Hmnhs}

In this section, using the adapted basis to non-horizontal submanifold $M$, we will give a new expression for the density of the measure $\mu$ on $M$. When $\mathbb G$ is the Heisenberg group, this measure is a constant multiple of the spherical Hausdorff measure. In the general case the density of the spherical Hausdorff measure on $M$ is a multiple of $\mu$ by a function. This function depends on the position of $TM$ in relation to the distribution $D$.

Following \cite{Magnani2008}, let $\tau_{M}(x)$ be a unit $(n-p)$-tangent vector of $M$ at $x$ with respect to the extended metric $h:=\langle,\rangle$ at $\mathbb G$. We define $\tau^{d}_{M}(x)$ as the part of $\tau_{M}(x)$ with maximum degree $d=Q-p$, where $Q$ is the Hausdorff dimension of $\mathbb G$.
Let  $v_1,\ldots,v_p$ be a basis of sections of the orthogonal space $TM^{\perp_h}$ which is orthonormal with respect to $h$, i.e., $\langle v_i,TM\rangle=0$, and  $v=v_1\wedge v_2\wedge\cdots\wedge v_p$ is the unit $p$-form of the orthogonal space $TM^{\perp_h}$. We consider the measure $\mu$ on $M$ with density
\begin{equation}
\dif \mu=|\tau^d_{M}|\mbox{vol}_{h}\llcorner M,
\end{equation}
where
$\mbox{vol}_{h}\llcorner M:=(v \ \lrcorner \ \mbox{dvol}_{h})|_{M}$,
the symbol $\lrcorner$ denotes the contraction (or interior product) of a differential form with a vector,
$
\mbox{dvol}_h=e^1\wedge\cdots\wedge e^n=f^1\wedge\cdots\wedge f^n
$
is the density of the spherical Hausdorff measure on $\mathbb G$ as a sub-Riemannian metric space,  and $\left|\cdot\right|$ is the norm induced by the fixed Riemannian metric $h$.
Then
\begin{equation*}
\dif \mu=|\tau^{d}_{M}|v \ \lrcorner \ \mbox{dvol}_{h}=|\tau^{d}_{M}|v \ \lrcorner \ f^{1}\wedge\cdots\wedge f^{n}.
\end{equation*}
\begin{thm}\label{dS}
\begin{equation}
\dif \mu=f^{p+1}\wedge\cdots\wedge f^n \ .\label{volume}
\end{equation}
\end{thm}
\begin{proof} First, we denote
\begin{equation*}
b_{jk}=\langle f_{j},f_{k}\rangle=\delta_{jk} + \sum^{p}_{\alpha=1}A^{\alpha}_{j}A^{\alpha}_{k}.
\end{equation*}
Observe that $\langle f_{\alpha},f_{k}\rangle=-A^{\alpha}_{k}$ for $j,k=d_{1}+1,\ldots ,n$ and $\alpha=1,\ldots,p$.
If $\mathbf B$ is the $(n-d_1)$-square matrix $\mathbf{B}=(b_{jk})$,  $\mathbf{A}$ is the $(n-d_1)\times p$ matrix
\begin{equation*}
\mathbf A=\left(\begin{array}{cccc}
A^{1}_{d_{1}+1}&A^{2}_{d_{1}+1}&\cdots&A^{p}_{d_{1}+1}\\
A^{1}_{d_{1}+2}&A^{2}_{d_{1}+2}&\cdots&A^{p}_{d_{1}+2}\\
\vdots &\vdots & \ddots &\vdots\\
A^{1}_{n}&A^{2}_{n}&\cdots&A^{p}_{n}
\end{array}\right)
\end{equation*}
and $\mathbf I_{n-d_1}$ is the $(n-d_1)$-identity matrix, we obtain
\begin{equation*}
\mathbf B=\mathbf I_{n-d_1}+\mathbf{A}\mathbf{A}^{T}.
\end{equation*}
Let us now calculate $v$. To do this, we project $f_1,\ldots,f_p$ on $TM^{\perp_h}$. These are
\begin{equation*}
w_{\alpha}:=f_{\alpha}+\sum^{n}_{j=d_{1}+1}x^{j}_{\alpha}f_{j},
\end{equation*}
where we choose $x^{j}_{\alpha}$ such that  $\langle w_{\alpha},f_{j}\rangle=0 $ for $\alpha=1,\ldots,p$, $j=p+1,\ldots,n$. In particular,
\begin{equation}\label{x}
\sum^{n}_{k=d_{1}+1}x^{k}_{\alpha}b_{jk}=A_j^\alpha.
\end{equation}
for $\alpha=1,\ldots,p$ and $j=d_1+1,\ldots,n$.  Thus the vectors $w_{\alpha}$, $\alpha=1,\ldots,p$, are linearly independent and orthogonal to $TM$. Hence, take
\begin{equation*}
v=\frac{ 1} {|w_{1}\wedge\cdots\wedge w_{p}|}w_{1}\wedge\cdots\wedge w_{p}.
\end{equation*}
Observe that $|w_{1}\wedge\cdots\wedge w_{p}|=|\det\mathbf{W}| $, where $\mathbf{W}=(W_{\alpha\beta})$ is a square $p$-matrix with
\begin{equation*}
W_{\alpha\beta}=\langle w_\alpha,w_\beta\rangle=\delta_{\alpha\beta}-\sum^{n}_{j=d_{1}+1}x^{j}_{\beta}A^{\alpha}_{j}.
\end{equation*}
If $c_{kj}$ are the entries of the inverse matrix of $\mathbf B$, it follows from (\ref{x}) that
\begin{equation}
x^{j}_{\alpha}=\sum^{n}_{k=d_{1}+1}c_{kj}A^{\alpha}_{k},\label{prodC}
\end{equation}
and therefore $W_{\alpha\beta}=\delta_{\alpha\beta}-\sum^{n}_{j=d_{1}+1}A^{\alpha}_{j}c_{jk}A^{\beta}_{k}$, or
$
\mathbf W=\mathbf I_p-\mathbf{A}^T\mathbf{B}^{-1}\mathbf{A}
$.
So,
\begin{equation*}
v\lrcorner\mbox{dvol}_h=\frac{w_{1}\wedge\cdots\wedge w_{p}}{|w_{1}\wedge\cdots\wedge w_{p}|}\,\lrcorner\,f^{1}\wedge\cdots\wedge f^{n}=\frac{1}{|\det\mathbf W|}f^{p+1}\wedge\cdots\wedge f^{n},
\end{equation*}
because $f^\alpha=0$ on $M$ for $\alpha=1,\ldots,p$.

Let us now find $\tau^{d}_{M}$. First
\begin{align*}
\tau_{M}&=|f_{p+1}\wedge\cdots\wedge f_{n}|^{-1}f_{p+1}\wedge\cdots\wedge f_{n}=|f_{d_1+1}\wedge\cdots\wedge f_{n}|^{-1}f_{p+1}\wedge\cdots\wedge f_{n}\\
&=|\det\mathbf B|^{-1}f_{p+1}\wedge\cdots\wedge f_{d_1}\wedge(e_{d_1+1}-\!\sum_{\alpha=1}^p\!A_{d_1+1}^\alpha f_{\alpha})\wedge\cdots\wedge (e_{n}-\!\sum_{\omega=1}^p\!A_{n}^\omega f_{\omega}).
\end{align*}
Therefore
\begin{align}\label{taudM}
\tau^{d}_{M}=|\det\mathbf B|^{-1}f_{p+1}\wedge\cdots\wedge f_{d_1}\wedge e_{d_1+1}\wedge\cdots\wedge e_{n},
\end{align}
where $d=d_1-p+2d_2+\ldots+ld_l=Q-p$. It follows that
$
|\tau^{d}_{M}|=|\det\mathbf B|^{-1}.
$
Then $$\dif\mu=|\det\mathbf B|^{-1}|\det\mathbf W|^{-1}f^{p+1}\wedge\cdots\wedge f^{n}\,.$$ To finish the proof, we must show that $|\det\mathbf B| |\det\mathbf W|=1$. To do this, we need only apply the matrix determinant lemma (see for instance \cite{Harville1998}). Taking determinants of
{\small{
\begin{equation*}
\left(\begin{array}{cc}
\mathbf I&\mathbf A^T\\
0&\mathbf B
\end{array}\right)
\left(\begin{array}{cc}
\mathbf I-\mathbf A^T\mathbf B^{-1}\mathbf A&0\\
\mathbf B^{-1}\mathbf A&\mathbf I
\end{array}\right)=\left(\begin{array}{cc}
\mathbf I&\mathbf A^{T}\\
\mathbf A& \mathbf B
\end{array}\right)=
\left(\begin{array}{cc}
\mathbf I&0\\
\mathbf A&\mathbf I
\end{array}\right)
\left(\begin{array}{cc}
\mathbf I&\mathbf A^{T}\\
0&\mathbf B-\mathbf A\mathbf A^T
\end{array}\right)
\end{equation*}}}
we have $\det(\mathbf B)\det(\mathbf W)=1$.
\end{proof}

\subsection{The Heisenberg group $\mathbb H^n$}\label{hspace}

We now will give a proof that the Heisenberg group has constant metric factor.

According to \cite{Magnani2008}, the metric factor $\theta(\tau)$ associated with a simple $s$-vector $\tau$ of $\wedge_s\mathfrak{g}$ is defined by
\begin{align}\label{teta}
\theta(\tau)=\mathcal{H}^s_{|\cdot|}(F^{-1}(\exp(\mathcal{L}(\tau))\cap B_1)),
\end{align}
where $F:\mathbb R^n\rightarrow\mathbb G$ is a system of graded coordinates with respect to an adapted orthonormal basis $(X_1,\ldots,X_n)$, $\mathcal{H}^s_{|\cdot|}$ is the $s$-dimensional Hausdorff measure with respect to the Euclidean norm of $\mathbb R^n$, $\mathcal L(\tau)$ is the unique subspace of $\mathfrak{g}$ associated to $\tau$ and $B_1$ is the open unit ball centered at $e$ with respect to a fixed homogeneous distance.

We now consider $\mathfrak h$ to be the Lie algebra of the Heisenberg group $\mathbb H^n$ generated by $e_1,\ldots, e_{2n}$, $e_{2n+1}$ with $[e_i,e_{i+n}]=e_{2n+1}$ for $i=1,\ldots,n$ and all the other brackets equal to $0$. We have $\mathfrak{h}=\mathfrak{h}^1\oplus\mathfrak{h}^2$, where $\mathfrak{h}^1$ is generated by $e_1,\ldots,e_{2n}$ and $\mathfrak{h}^2$ by $e_{2n+1}$. Take a scalar product on $\mathfrak{h}^1$ such that $e_1,\ldots,e_{2n}$ is an orthonormal basis. We extend the chosen scalar product to $\mathfrak h$ so that $e_{2n+1}$ is unitary and orthogonal to $\mathfrak{h}^1$.

The group operation of Heisenberg group $\mathbb H^n$ is given in exponential coordinates by
\begin{align*}
(x^1,\ldots,x^{2n},x^{2n+1})(y^1,\ldots,y^{2n},y^{2n+1})=&\Big( x^1+y^1,\ldots,x^{2n}+y^{2n},x^{2n+1}+y^{2n+1}\\
&+\frac{1}{2}\sum_{i=1}^n(x^iy^{i+n}-y^ix^{i+n})\Big)
\end{align*}
and the left-invariant vector fields are
\begin{equation*}
e_i=\frac{\partial}{\partial x^i}-\frac{1}{2}x^{i+n} \frac{\partial}{\partial x^{2n+1}},\,
e_{i+n}=\frac{\partial}{\partial x^{i+n}}+\frac{1}{2}x^{i} \frac{\partial}{\partial x^{2n+1}},\,
e_{2n+1}=\frac{\partial}{\partial x^{2n+1}}.
\end{equation*}

We shall first find a parametrization for the open unit ball $B_1$ of $\mathbb H^n$ whose expression also works for a parametrization of $\exp ^{-1}(B_1) \subset\mathfrak{h}$.

By \cite[Proposition~5.1]{Diniz2009}, the Carnot-Caratheodory geodesics  $c(t)$ passing by the identity are solutions of the following system
\begin{align*}
\dot{c}=\sum_{r=1}^{2n}\lambda_re_r,\quad \dot{\lambda}_j=-\lambda_0\lambda_{j+n},\quad \dot{\lambda}_{j+n}=\lambda_0\lambda_{j},\quad
\dot{\lambda}_0=0.
\end{align*}
with initial conditions $c(0)=(0,\ldots,0)$ and $\lambda_k(0)=\mu_k$, $k=0,\ldots,2n$. Therefore,
\begin{align*}
\lambda_0&=\mu_0\\
\lambda_j&=\mu_j\cos\left(\mu_0 t\right)-\mu_{j+n}\sin\left(\mu_0 t\right)\\
\lambda_{j+n}&=\mu_{j+n}\cos\left(\mu_0 t\right)+\mu_{j}\sin\left(\mu_0 t\right)
\end{align*}
and
\begin{align*}
\dot{x}^j&=\mu_j\cos\left(\mu_0 t\right)-\mu_{j+n}\sin\left(\mu_0 t\right)\\
\dot{x}^{j+n}&=\mu_{j+n}\cos\left(\mu_0 t\right)+\mu_{j}\sin\left(\mu_0 t\right)\\
\dot{x}^{2n+1}&=\frac{1}{2}\sum_{j=1}^n(\dot{x}^{j+n}x^j-\dot{x}^jx^{j+n}).
\end{align*}
Integrating from $0$ to $t$ we get
\begin{align*}
{x}^j(\mu_0,\ldots,\mu_{2n},t)&=\frac{\mu_j}{\mu_0}\sin\left(\mu_0 t\right)-\frac{\mu_{j+n}}{\mu_0}\left(1-\cos\left(\mu_0 t\right)\right)\\
{x}^{j+n}(\mu_0,\ldots,\mu_{2n},t)&=\frac{\mu_j}{\mu_0}\left(1-\cos\left(\mu_0 t\right)\right)+\frac{\mu_{j+n}}{\mu_0}\sin\left(\mu_0 t\right)\\
{x}^{2n+1}(\mu_0,\ldots,\mu_{2n},t)&=\frac{1}{2\mu_0^2}\left(\sum_{r=1}^{2n}\mu_r^2\right)\left(\mu_0 t-\sin\left(\mu_0 t\right)\right).
\end{align*}
Hence, the unitary ball $B_1$ defined by the homogeneous distance $\rho$ is parameterized by the Carnot-Caratheodory exponential as
\begin{equation*}
\exp_{CC}(\mu_0,\ldots,\mu_{2n})=\left({x}^1(\mu_0,\ldots,\mu_{2n},1),\ldots,{x}^{2n+1}(\mu_0,\ldots,\mu_{2n},1)\right)
\end{equation*}
on the cylinder $-2\pi\leq\mu_0\leq2\pi$, $\sum_{r=1}^{2n}\mu_r^2\leq 1$. Observe that
\begin{equation*}
\left[\begin{array}{l}
{x}^j(\mu_0,\ldots,\mu_{2n},1)\\
{x}^{j+n}(\mu_0,\ldots,\mu_{2n},1)
\end{array}\right]=\frac{1}{\mu_0}
\left[
\begin{array}{lr}
\sin(\mu_0)&-(1-\cos(\mu_0))\\
(1-\cos(\mu_0))&\sin(\mu_0)
\end{array}
\right]
\left[
\begin{array}{l}
\mu_j\\
\mu_{j+n}
\end{array}
\right].
\end{equation*}
We transform the matrix on the right-hand side into an  orthogonal matrix as follows:
\begin{equation*}
\left[\begin{array}{l}
{x}^j(\mu_0,\ldots,\mu_{2n},1)\\
{x}^{j+n}(\mu_0,\ldots,\mu_{2n},1)
\end{array}\right]=\frac{\sqrt{2(1-\cos(\mu_0))}}{|\mu_0|}
\left[
\begin{array}{lr}
\cos\alpha&-\sin\alpha\\
\sin\alpha&\cos\alpha
\end{array}
\right]
\left[
\begin{array}{l}
\mu_j\\
\mu_{j+n}
\end{array}
\right]
\end{equation*}
where $\cos\alpha=\dfrac{|\sin \mu_0|}{\sqrt{2(1-\cos\mu_0)}}$ and $\sin\alpha=\dfrac{\sqrt{2(1-\cos\mu_0)}}{2\sgn(\mu_0)}$. If we introduce the change of variables
\begin{equation*}
\left[
\begin{array}{l}
\overline\mu_j\\
\overline\mu_{j+n}
\end{array}
\right]
=
\left[
\begin{array}{lr}
\cos\alpha&-\sin\alpha\\
\sin\alpha&\cos\alpha
\end{array}
\right]
\left[
\begin{array}{l}
\mu_j\\
\mu_{j+n}
\end{array}
\right],
\end{equation*}
then we can parameterize $B_1$ by
\begin{align}\label{par}
\mathcal{P}(\mu_0,\overline\mu_1\ldots,\overline\mu_{2n})=\Big(\varepsilon(\mu_0)\overline\mu_1,\ldots,\varepsilon(\mu_0)\overline\mu_{2n}, \frac{\mu_0 -\sin\mu_0}{2\mu_0^2}\sum_{r=1}^{2n}\overline\mu_r^2\Big)\,.
\end{align}
on the cylinder $-2\pi\leq\mu_0\leq2\pi$, $\sum_{r=1}^{2n}\overline\mu_r^2\leq 1$, where $\varepsilon(\mu_0)= \dfrac{\sqrt{2(1-\cos\mu_0})}{|\mu_0|}$.

Using \eqref{par} we can prove the result following:
\begin{thm}
If $M$ is a non-horizontal submanifold of $\mathbb H^n$, then  the $\mu$-measure on $M$ is a constant multiple of the spherical Hausdorff measure.
\end{thm}
\begin{proof} According to formula \eqref{M} and \eqref{teta} we must show that
\begin{align*}
\theta(\tau_M^d(x))=\mathcal{H}^{2n+1-p}_{|\cdot|}(F^{-1}(\exp(\mathcal{L}(\tau_M^d(x)))\cap B_1))
\end{align*}
is constant on a non-horizontal submanifold $M$ of $\mathbb{H}^n$.

From \eqref{taudM} we get

\begin{equation*}
\tau^{d}_{M}=|\det\mathbf B|^{-1}f_{p+1}\wedge\cdots\wedge f_{2n}\wedge e_{2n+1}.
\end{equation*}
Therefore, ${\mathcal L}(\tau^d_M)$ is generated by $f_{p+1},\ldots,f_{2n},e_{2n+1}$.

Let $A:\mathbb R^{2n+1}\rightarrow \mathfrak h$ be the linear transformation such that $F:\mathbb R^{2n+1}\rightarrow \mathbb H^{2n+1}$ is given by $F=\exp\circ A$. Then,
\begin{align*}
F^{-1}(\exp ({\mathcal L}(\tau_M^d))\cap B_1)=A^{-1}({\mathcal L}(\tau_M^d)\cap \exp^{-1}(B_1)).
\end{align*}
Let $\phi\in O(2n+1,\mathbb{R})$ such that $\phi(0,\ldots,0,1)=(0,\ldots,0,1)$ and $\tilde \phi=A\phi A^{-1}$.
Note that $\tilde\phi$ acts in $\mathfrak h$ as an isometry, fixes $e_{2n+1}$ and ${\mathcal L}(\tilde \phi\tau_M^d)=\tilde \phi{\mathcal L}(\tau_M^d)$.

It follows from formula (\ref{par}) that $\tilde \phi(\exp^{-1}(B_1))=\exp^{-1}(B_1)$. Hence,
\begin{align*}
\theta(\tilde \phi\tau_M^d)&={\mathcal H}^{2n+1-p}_{|.|}\Big(F^{-1}\Big(\exp({\mathcal L}(\tilde \phi\tau_M^d))\cap B_1\Big)\Big)\\
&={\mathcal H}^{2n+1-p}_{|.|}\Big(A^{-1}\Big({\mathcal L}(\tilde \phi\tau_M^d)\cap\exp^{-1}(B_1)\Big)\Big)\\
&={\mathcal H}^{2n+1-p}_{|.|}\Big(A^{-1}\Big(\tilde \phi {\mathcal L}(\tau_M^d)\cap \tilde \phi\Big(\exp^{-1}(B_1)\Big)\Big)\Big)\\
&={\mathcal H}^{2n+1-p}_{|.|}\Big(A^{-1}\circ \tilde \phi\Big({\mathcal L}(\tau_M^d)\cap \exp^{-1}(B_1)\Big)\Big)\\
&={\mathcal H}^{2n+1-p}_{|.|}\Big(\phi\circ A^{-1}\Big({\mathcal L}(\tau_M^d)\cap \exp^{-1}(B_1)\Big)\Big)\\
&={\mathcal H}^{2n+1-p}_{|.|}\Big(A^{-1}\Big({\mathcal L}(\tau_M^d)\cap \exp^{-1}(B_1)\Big)\Big)\\
&=\theta(\tau_M^d).
\end{align*}
Since, for any $y\in M$, ${\mathcal L}(\tau^{d}_{M}(y))$ can be obtained by the action of some $\tilde\phi$ on ${\mathcal L}(\tau^{d}_{M}(x))$, the function $\theta(\tau^{d}_{M})$ is constant on $M$.

\end{proof}

\section{First variation of $\mu$}\label{vvf}
Consider an immersion $i:M\rightarrow\mathbb G$ such that $i(M)$ in $\mathbb G$ is a non-horizontal submanifold of codimension $p$.  On $\mathbb G$ we have the volume element
\begin{equation}\label{volumeG}
\dif V=e^1\wedge\cdots\wedge e^n=f^1\wedge\cdots\wedge f^n
\end{equation}
and on $i(M)$ the volume form
\begin{equation*}
\tilde\Gamma(0):=\dif \mu=f^{p+1}\wedge\cdots\wedge f^n.
\end{equation*}
Let $F:(-\epsilon,\epsilon)\times M\rightarrow\mathbb{G}$ be a differentiable map such that $F_u:M\rightarrow\mathbb{G}$ is an immersion for all $u\in(-\epsilon,\epsilon)$, where $F_u(p)=F(u,p)$, and $i=F_0$. The map $F$ is a variation of $i$.  On each submanifold $F_u(M)$ we have the volume form $\tilde\Gamma(u)$ as constructed in Section \ref{Hmnhs}. Thus, we have a family of volume forms $\Gamma(u)=F_u^*(\tilde\Gamma(u))$ on $M$ and by the arguments of \cite[p.~286]{Spivak1979}  we have
\begin{align}\label{leibniz}
\left.\frac{\dif}{\dif u}\right|_{u=u_0}\int_M\Gamma(u)=\int_M\left.\frac{\dif}{\dif u}\right|_{u=u_0}\Gamma(u)\,.
\end{align}
We now state the main result of this section.
\begin{thm} \label{firstvar}
Let $i:M\rightarrow\mathbb G$ be an immersion of an oriented $(n-p)$-dimensional manifold $M$ into a sub-Riemannian stratified Lie group $(\mathbb G,D,\langle\, ,\,\rangle)$ as a non-horizontal submanifold and let $F:(-\epsilon,\epsilon)\times M \rightarrow\mathbb{G}$ be a variation of $i$ through immersions with variation vector field $W$. If $\Gamma(u)=F_u^*(\tilde\Gamma(u))$, where $\tilde\Gamma(u)$ is the volume form of the $\mu$-measure in $F_u(M)$ in accordance with the given orientation of $M$, then
\begin{equation*}
\dot{\Gamma}(0)=i^*\left(
H_{ W^\perp}+\sigma_{ W^\perp}\right)\Gamma(0)+\dif \left(i_*^{-1}(W^\top)\lrcorner\Gamma(0)\right),
\end{equation*}
where $H$ is the mean curvature and $\sigma$ is the mean torsion of $i(M)$.
\end{thm}
\begin{proof} We follow the ideas of Spivak \cite{Spivak1979}. Let $W=\left.\frac{\dif}{\dif u}F\right|_{u=0}$ be the variation vector field. Choose a point $p_0\in M$ such that $W_{p_0}=W(i(p_0))$ is not tangent to $i(M)$. Then, in a neighborhood $O$ of $p_0$ (and decreasing $\epsilon$ if necessary) we can suppose that $F:(-\epsilon,\epsilon)\times O\rightarrow\mathbb G$ is an embedding. Let us denote $F_u(O)=O_u$. On an open set $\widetilde O$ of $\mathbb G$ containing the image of $(-\epsilon,\epsilon)\times O$ we can choose a basis of vector fields $f_1,\ldots,f_n$ such that
\begin{enumerate}
\item [(i)] $f_1,\ldots,f_p$ restricted to $O_u$ are in $TO_u^\perp$;
\item [(ii)] $f_{p+1},\ldots,f_n$ restricted to $O_u$ are in $TO_u$;
\item [(iii)] $f_1,\ldots,f_{d_1}$ are orthonormal in $D$;
\item [(iv)] $f_i=e_i-\sum_{\alpha=1}^pA_i^\alpha f_\alpha$,  for $i=d_1+1,\ldots,n$.
\end{enumerate}
If $f^1,\ldots,f^n$ are the dual $1$-forms, then $F_u^*(f^\alpha)=0$ for $\alpha=1,\ldots,p$. Furthermore, if
$
\Phi=f^{p+1}\wedge\cdots\wedge f^n,
$
 we get $\tilde \Gamma(u)=\Phi|_{O_u}$ and
\begin{equation*}
\Gamma(u)=F_u^*\tilde \Gamma(u)=F_u^*\Phi=F_u^*(f^{p+1}\wedge\ldots\wedge f^n)
\end{equation*}
 is a volume form on $M$.

Now, the vector field $\frac{\dif}{\dif u}F$ defined on $F((-\epsilon,\epsilon)\times O)$ can be extended further to a vector field $\widetilde W$ defined on the open set $\widetilde O$. Associated to $\widetilde W$ there is a local $1$-parameter group of diffeomorphisms $\rho_u$ such that $\rho_u(F_v(p))=F_{u+v}(p)$. Hence, if $X$ is a tangent vector field on $O$, then $(\rho_u)_*(F_v)_*X=(F_{u+v})_*X$. Therefore,

\begin{equation}\label{Fu}
\dot{\Gamma}(u)=F_u^*L_{\widetilde W}\Phi=F_u^*\left(\widetilde W\lrcorner\dif\Phi+\dif \widetilde W\lrcorner\Phi\right),
\end{equation}
where $L_{\widetilde W}$ is Lie derivative with respect to $\widetilde W$.

On the other hand, using the Proposition \ref{stru} we obtain
\begin{align*}
\dif\Phi=&\sum_{j=p+1}^n(-1)^{j-(p+1)}f^{p+1}\wedge\cdots\wedge\dif f^j\wedge\cdots\wedge f^n\\
=&\sum_{j=p+1}^{d_1}(-1)^{j-(p+1)}f^{p+1}\wedge\cdots\wedge\left(\sum^n_{i=1}f^{i}\wedge\overline{\omega}_i^{j}\right)\wedge\cdots\wedge f^n\\
&+\sum_{j=d_1+1}^n(-1)^{j-(p+1)}f^{p+1}\wedge\cdots\wedge \overline T^j\wedge\cdots\wedge f^n
\end{align*}
Since $\overline T^j=\frac{1}{2}\sum_{k,l=1}^n\overline T^j(f_k,f_l)f^k\wedge f^l$, then we can write the above equality as follows.
\begin{align*}
\dif\Phi=&\sum_{\alpha=1}^p\left(\sum^{d_1}_{i=p+1}\overline{\omega}^{i}_{\alpha}(f_i)+\sum^n_{j=d_1+1}\overline T^j(f_\alpha,f_j)\right)f^\alpha\wedge\Phi\\
&+\sum_{\alpha,\beta=1}^p\left(\sum^{d_1}_{i=p+1}\overline{\omega}^{i}_{\alpha}(f_\beta)+\frac{1}{2}\sum^n_{j=d_1+1}\overline T^j(f_\alpha,f_\beta)\right)f^\alpha\wedge\Psi\\
=&\sum^p_{\alpha=1}\Big(H_{f_\alpha}+\sigma_{f_\alpha}\Big)f^\alpha\wedge\Phi+\sum_{\alpha,\beta=1}^p\Omega_{\alpha\beta}f^\alpha\wedge\Psi,
\end{align*}
where
\begin{equation*}
\Psi=f^{p+1}\wedge\cdots\wedge f^\beta\wedge\cdots\wedge f^n\,\, \mbox{and}\,\, \Omega_{\alpha\beta}=\sum^{d_1}_{i=p+1}\overline{\omega}^{i}_{\alpha}(f_\beta)+\frac{1}{2}\sum^n_{j=d_1+1}\overline T^j(f_\alpha,f_\beta).
\end{equation*}
It now follows that
\begin{align*}
\widetilde W\lrcorner\dif\Phi=&\sum^p_{\alpha=1}\Big(H_{f_\alpha}+\sigma_{f_\alpha}\Big)\Big(\widetilde W\lrcorner(f^\alpha\wedge\Phi)\Big)+\sum_{\alpha,\beta=1}^p\!\!\Omega_{\alpha\beta}\Big(\widetilde W\lrcorner(f^\alpha\wedge\Psi)\Big)\\
=&\sum^p_{\alpha=1}\Big(H_{f_\alpha}+\sigma_{f_\alpha}\Big)\Big(f^\alpha(\widetilde W)\Phi-f^\alpha\wedge(\widetilde W\lrcorner\Phi)\Big)\\
&+\sum_{\alpha,\beta=1}^p\Omega_{\alpha\beta}\Big(f^\alpha(\widetilde W)\Psi-f^\alpha\wedge(\widetilde W\lrcorner\Psi)\Big).
\end{align*}
Since $F_u^*f^\alpha=0$ and $F_u^*\Psi=0$, then
\begin{equation*}
F_u^*(\widetilde W\lrcorner\dif\Phi)=F_{u}^*\left(\sum_{\alpha=1}^pf^\alpha(\widetilde W)(H_{f_\alpha}+\sigma_{f_\alpha})\right)\Gamma(u)=F_{u}^*\left(H_{\widetilde W^\perp}+\sigma_{\widetilde W^\perp}\right)\Gamma(u) \ .
\end{equation*}
For the other term is proved easily that
$
\dif (\widetilde W\lrcorner\Phi)=\dif(\widetilde W^\top\lrcorner\Phi).
$

Now, equation (\ref{Fu}) implies
\begin{equation*}
F_u^*L_{\widetilde W}\Phi=F_{u}^*\left(H_{\widetilde W^\perp}+\sigma_{\widetilde W^\perp}\right)\Gamma(u)+F_{u}^*(\dif(\widetilde W^\top\lrcorner\Phi)) \ .
\end{equation*}
By setting $u=0$ we obtain the proof of this theorem at any point $p_0$ for which $W(p_0)$ is not tangent to $i(M)$.
In the general case, we proceed as follows.

Let ${\bf{G}}=\mathbb{G}\times\mathbb{R}$ be the Lie group whose Lie algebra ${\bf{g}}=\mathfrak{g}\oplus {\mathbb{R}}$ (where $\mathfrak{g}$ is the Lie algebra of $\mathbb{G}$) admits a stratification
${\bf{g}}={\bf{g^1}\oplus g^2\oplus\cdots\oplus g^r}$ such that
\begin{equation*}
{\bf{g^1}}=\mathfrak{g}^1\oplus\mathbb{R}, \ {\bf g^i}=\mathfrak{g}^i, \ \mbox{for} \ i=2,3,\ldots,r \ \mbox{and} \ [\mathfrak{g}^i,\mathbb{R}]=0 \ \forall \ i .
\end{equation*}
Let $\bf{D}$ be the distribution on $\bf{G}$ generated by $\bf{g}^1$. Therefore, $(\bf{G, D},\langle,\rangle)$ is a sub-Riemannian stratified Lie group, where  by $\left\langle,\right\rangle$ we denote the scalar product induced on $\bf{G}$.

Define ${\widehat{F}}:(-\epsilon,\epsilon)\times M\rightarrow{\bf{G}}$ by
$
{\widehat{F}}(u,p)=(F(u,p),u).
$
The new variation vector field $\bf{W}$ is
$
{\bf{W}}(p)=(W(p),1),
$
where $1$ denotes the unit vector field on $\mathbb{R}$. Observe that $\bf{W}$ is not tangent to $\widehat{F}_0(M)$, so the theorem holds for $\widehat{F}$. On the other hand
\begin{equation*}
{\bf{A_{\widetilde{W}^{\perp}}}}({\bf{X}})=(A_{\widetilde{W}^{\perp}}(X),0), \ {\bf\overline{T}(\widetilde{W}^\perp,X)}=(\overline{T}(\widetilde{W}^\perp,X),0)
\end{equation*}
where ${\bf{X}}=X+\xi\in T{\bf{G}}=T\mathbb G\oplus\mathbb{R}$, which proves the claim of theorem for $F$.

\end{proof}

\begin{cor}\label{eqmin}
Let $M$ be an oriented  compact manifold  of dimension $n-p$ with boundary, $\mathbb G$ a stratified Lie group of dimension $n$ and $i:M\rightarrow\mathbb G$  an immersion  of $M$ in $\mathbb G$ as a non-horizontal submanifold. Let $F:(-\epsilon,\epsilon)\times M \rightarrow\mathbb G$ be  a variation of $i$ through immersions, with variation vector field $W$. If $\Gamma(u)=F_u^*(\tilde\Gamma(u))$, where $\tilde\Gamma(u)$ is the volume form in $F_u(M)$ in accordance with the given orientation of $M$, and
\begin{equation*}
V(u)=\int_M\Gamma(u)
\end{equation*}
then
\begin{equation*}
\left.\frac{\dif}{\dif u}V(u)\right|_{u=0}=\int_M i^*\left(
H_{ W^\perp}+\sigma_{ W^\perp}\right)\Gamma(0)+\int_{\partial M} i_*^{-1}(W^\top)\lrcorner\Gamma(0),
\end{equation*}
where $H$ is the mean curvature and $\sigma$ the mean torsion of $i(M)$. In particular, if $F$ is a variation that keeps the boundary of $M$ fixed, then
\begin{equation*}
\left.\frac{\dif}{\dif u}V(u)\right|_{u=0}
=\int_M i^*\left(
H_{ W^\perp}+\sigma_{ W^\perp}\right)\Gamma(0)\,.
\end{equation*}
The immersion $i$ is a critical point for $V$ among all immersions $f:M\rightarrow\mathbb G$ with $f=i$ on $\partial M$ if and only if
\begin{equation*}
H+\sigma=0
\end{equation*}
on $TM^\perp$.
\end{cor}
\begin{proof} The first claim follows from Theorem \ref{firstvar}, formula (\ref{leibniz}) and the Stokes'  theorem. If $F$ keeps $\partial M$ fixed, then $W=0$ on $M$, so $i_*^{-1}(W^\top)\lrcorner\Gamma(0)$ on $\partial M$, which proves the second statement. The third one follows from the choice of a vector field $v\in TM^\perp$ such that for every $w\in TM^\perp$, $H_w+\sigma_w=\langle v,w\rangle$ and a positive function $\phi$ on the interior of $M$ such that $\left.\phi\right|_{\partial M}=0$. The last statement follows by taking $W^\perp=\phi v$.
\end{proof}
\begin{cor}\label{eqminhiper}
Let the assumptions of Corollary \ref{eqmin} hold. If $M$ is a hypersurface,
then
\begin{equation*}
\left.\frac{\dif}{\dif u}V(u)\right|_{u=0}
=\int_M i^*\left(
H_{ W^\perp}\right)\Gamma(0)+\int_{\partial M} i_*^{-1}(W^\top)\lrcorner\Gamma(0),
\end{equation*}
where $H$ is the mean curvature of $i(M)$. In particular, if $F$ is a variation that keeps the boundary of $M$ fixed, then
\begin{equation*}
\left.\frac{\dif}{\dif u}V(u)\right|_{u=0}
=\int_M i^*\left(
H_{ W^\perp}\right)\Gamma(0),
\end{equation*}
The immersion $i$ is a critical point for $V$, among all immersions $f:M\rightarrow\mathbb G$ with $f=i$ on $\partial M$ if and only if
$
H=0
$
on $TM^\perp$.
\end{cor}
\begin{proof} Observe that
\begin{equation*}
\overline T^j(f_\alpha,f_j)=\overline T^j(f_\alpha,e_j-\sum_{\beta=1}^p A_j^\beta f_\beta)=\sum_{i=1}^{d_1} a_\alpha^i\overline T^j(e_i,e_j)-\sum_{\beta=1}^p A_j^\beta\overline T^j(f_\alpha,f_\beta)
\end{equation*}
and as $c_{ij}^j=0$ we get
\begin{equation*}
\overline T^j(f_\alpha,f_j)=-\sum_{\beta=1}^p A_j^\beta\overline T^j(f_\alpha,f_\beta)
\end{equation*}
 and $\overline T^j(f_\alpha,f_\beta)=0$ if $\deg j\neq 2$. In particular, if $M$ is a hypersurface of $\mathbb G$, then $p=1$ and we obtain $\overline T^j(f_1,f_j)=0$ for every $j=d_1+1,\ldots,n$.
 \end{proof}
A special  case is when $M$ has the same dimension of $\mathbb G$, so that $M$ is a compact $n$-dimensional manifold with boundary in $\mathbb G$. Then $T_pM=T_p\mathbb G$ so $TM^\perp=0$ for all $p\in M$. Consequently, $H_{ W^\perp}+\sigma_{ W^\perp}=0$ and
\begin{equation*}
\left.\frac{\dif}{\dif u}V(u)\right|_{u=0}=\int_{\partial M} i_*^{-1}(W)\lrcorner\Gamma(0),
\end{equation*}
for any $W\in TM$.
If $\partial M$ is non-horizontal in all of its points, we can choose an adapted  basis $f_1,\ldots,f_n$ of $T\mathbb G|_{\partial M}$. Then
$ i_*^{-1}(W)\lrcorner\Gamma(0)=f^1(W)f^2\wedge\cdots\wedge f^n$. Since $f^1=0$ on $T\partial M$ we have that
\begin{equation*}
\left.\frac{\dif}{\dif u}V(u)\right|_{u=0}=\int_{\partial M}f^1(W)\dif\mu,
\end{equation*}
where $\dif\mu=f^2\wedge\cdots\wedge f^n$.

Motivated by Corollary \ref{eqmin} we give the following definition.
\begin{dfn}\label{Hmin}
We say that a non-horizontal submanifold is \emph{minimal} if
\begin{equation*}
H+\sigma=0\,\, \mbox{on}\,\, TM^\perp.
\end{equation*}
\end{dfn}

\begin{ex}\label{verticalH}
Let $M\subset\mathbb{R}^{2n}$ be a minimal submanifold and $N\subset\mathbb H^n$ non-horizontal submanifold defined by  $N=\{(x,t)\in \mathbb H^n: x\in M, t\in \mathbb R\}$. Then, $N$ is minimal.
\end{ex}
Let  $\pi:
\mathbb H^n\rightarrow \mathbb R^{2n}$ be the natural projection and let $\nabla^{\mathbb R}$ be the canonical Riemannian connection on $\mathbb R^{2n}$.
 Let $g_1,\ldots,g_p$ and $g_{p+1},\ldots,g_{2n}$ be orthonormal basis of $TM^{\perp}$ and $TM$, respectively. Let $f_1,\ldots, f_{2n}$ be a basis of $D$ restricted to $N$ such that $\pi_*(f_j)=g_j$ for $j=1,\ldots, 2n$. Therefore, $f_{2n+1}=\dfrac{\partial}{\partial x^{2n+1}}$ extends $f_{p+1},\ldots, f_{2n}$ to a basis of $TN$. Clearly, $\pi_*(\overline\nabla_{f_i} f_\alpha)=\nabla^{\mathbb R}_{g_i}g_\alpha$ for $i=p+1,\ldots, 2n$ and $\alpha=1,\ldots,p$.
As $
\overline\nabla f_\alpha=\sum_{i=p+1}^{2n}\overline\omega_\alpha^i\otimes f_i,
$
and if $
\nabla^{\mathbb R} g_\alpha=\sum_{i=p+1}^{2n}\psi_\alpha^i\otimes g_i
$ for $j=p+1,\ldots, 2n$, we have that $\pi^*\psi_\alpha^i=\overline{\omega}_\alpha^i.$ It follows from (\ref{A}) that
\begin{equation*}
\pi_{*}(A_{f_\alpha}(f_i))=-\sum_{j=p+1}^{2n}\pi_{*}(\overline\omega_\alpha^j(f_i)f_j)=-\sum_{j=p+1}^{2n}\psi_\alpha^j(g_i)g_j=A_{g_\alpha}(g_i).
\end{equation*}
Because $M$ is minimal, then
\begin{align}\label{AW}
\sum_{j=p+1}^{2n}\langle A_{f_\alpha}(f_j),f_j\rangle=0.
\end{align}
Now, for the mean torsion we have that
\begin{align}\label{TW}
 \overline T^{2n+1}(f_\alpha,f_{2n+1})=\sum_{k=1}^{2n}a_\alpha^k \overline T^{2n+1}(e_k,e_{2n+1})=0,
\end{align}
because $ \overline T(e_k,e_{2n+1})=-[e_k,e_{2n+1}]=0$ for $k=1,\ldots,2n$. Hence, from (\ref{AW}) and (\ref{TW}), $N$ is minimal.

The next example is well known in the literature (see \cite{Pansu1982}).
\begin{ex} Let $S$ be a non-horizontal surface in $\mathbb H^1$ at all points. If $S$ is minimal surface, then $S$ is ruled.
\end{ex}
By Corollary \ref{eqmin}, we have
$\langle A_{f_1}(f_2),f_2\rangle=0$ or $\langle \overline\nabla_{f_2}f_1,f_2\rangle=0$.
It follows that
$\langle \overline\nabla_{f_2}f_2,f_1\rangle=0$, and therefore $\overline\nabla_{f_2}f_2=0$.
Consider $\pi$ and $\nabla^\mathbb R$ as in the Example \ref{verticalH}.  Let $\gamma$ be a tangent horizontal curve to $S$, with $\gamma'(t)=f_2(\gamma(t))$ and $g(t)=\pi(\gamma(t))$ the projection of $\gamma$ on $\mathbb R^2$. Then, $|g'(t)|=|\pi_*\gamma'(t)|=1$ and
\begin{equation*}
\nabla^\mathbb R_{g'(t)}g'(t)=\pi_*(\overline\nabla_{f_2}f_2)=0.
\end{equation*}
Therefore, the curvature of curve $g$ is $0$. It follows that $g$ is a line in $\mathbb R^2$. If  $g(t)=(x_0+at,y_0+bt)$, then
\begin{equation*}
\gamma(t)=(x_0+at,y_0+bt,z_0+\frac{1}{2}(bx_0-ay_0)t)
\end{equation*}
is a line, and it follows that $S$ is ruled.

\section{Non-horizontal surfaces in $\mathbb{H}^2$}\label{H2}
	
In this section we give new examples of minimal non-horizontal surfaces in the 5-dimensional  Heisenberg group $\mathbb{H}^2$.

The Lie algebra $\mathfrak{h}=\mathfrak{h}^1\oplus\mathfrak{h}^2$ of the Heisenberg group $\mathbb{H}^2$ is generated by $e_1, e_2, e_3, e_4, e_5$ with $[e_1,e_3]=[e_2,e_4]=e_5$. Let us define two linear operators acting on $\mathfrak{h}^1$. The first operator we define by $Je_1=e_3$, $Je_2=e_4$, $Je_3=-e_1$ and $Je_4=-e_3$.
The second we define by $Re_1=e_2$, $Re_2=-e_1$, $Re_3=-e_4$ and $Re_4=e_3$. Consider the scalar product in $\mathfrak{h}^1$ such that $e_1, e_2, e_3, e_4$ is an orthonormal basis of $\mathfrak{h}^1$. One can easily prove the following properties of $J$ and $R$.
\begin{prop}\label{JR} The following relations for $J$ and $R$ are true:
\begin{enumerate}
\item $J^2=R^2=-\mbox{Id}$;
\item $RJ+JR=0$.
\item $\langle Rx,Ry\rangle=\langle Jx,Jy\rangle=\langle x,y\rangle, \ \forall \ x,y\in\mathfrak{h}^1$;
\item $\langle Jx,Ry\rangle+\langle Rx,Jy\rangle=0,\forall \ x,y\in\mathfrak{h}^1$;
\item $[x,Jy]=\langle x,y\rangle e_5,\, [x,Ry]=\langle Jx,Ry\rangle e_5,\, \forall \ x,y\in\mathfrak{h}^1$,
\end{enumerate}
where $\mbox{Id}$ is identity application of $\mathfrak h^1$.
\end{prop}
 \begin{prop}\label{basissuph4}
If $v_4\in\mathfrak{h}_1$ is unitary vector and $v_3=Rv_4,\, v_2=-Jv_4,\, v_1=-Rv_2$, then $v_1, v_2, v_3, v_4$ is orthonormal basis of $\mathfrak{h}^1$ such that the only non-null brackets are $
[v_1,v_3]=[v_2,v_4]=e_5$.
\end{prop}
\begin{proof} It follows immediately from  Proposition \ref{JR} that $v_1,v_4,v_3,v_4$ is an orthonormal basis of $\mathfrak{h}^1$.
Moreover,
\begin{align*}
[v_1,v_3]&=-[Rv_2,Rv_4]=-\langle JRv_2,Rv_4\rangle e_5=\langle RJv_2,Rv_4\rangle e_5=\langle Jv_2,v_4\rangle e_5=e_5\\
[v_2,v_4]&=[-Jv_4,v_4]=\langle v_4,v_4\rangle e_5=e_5,\\ [v_2,v_3]&=[-Jv_4,Rv_4]=\langle Rv_4,v_4\rangle e_5=0,\\ [v_1,v_4]&=-[Rv_2,v_4]=\langle Jv_4,Rv_2\rangle e_5=-\langle v_2,Rv_2\rangle e_5=0,\\
[v_1,v_2]&=[Rv_2,Jv_4]=\langle Rv_2,v_4\rangle e_5 =-\langle v_1,v_4\rangle e_5=0,\\
[v_3,v_4]&=[ Rv_4,v_4]=-\langle Jv_4,Rv_4\rangle=0\,.
\end{align*}
\end{proof}
 Naturally, we can extend the operators $J$ and $R$ to $D$ and it satisfy $\overline \nabla J=0$ and $\overline\nabla R=0$. We then get the following consequence of Proposition \ref{basissuph4}.
\begin{cor}\label{baseD}
If $f_4\in D$ is an unitary vector field and $f_3=Rf_4, f_2=-Jf_4, f_1=-Rf_2$, then $f_1,f_2,f_3,f_4$ is an orthonormal basis of $D$ such that $\overline T^5(f_1,f_3)=\overline T^5(f_2,f_4)=-1$ and others are null.
\end{cor}

Let $M\subset \mathbb H^2$ be a non-horizontal submanifold of dimension two. Then $TM\cap D$ is of dimension one. Choose a unitary vector field $f_4 $ in $TM\cap D$. Using Corollary \ref{baseD} we complete it to an orthonormal basis of $D$ by choosing $f_3=Rf_4$, $f_2=-Jf_4$, $f_1=-Rf_2$ in $TM^\perp$. Moreover, we get for every $X\in TM$
\begin{equation*}
\overline\omega_4^1(X)=\overline\omega_3^2(X), \overline\omega_4^2(X)=-\overline\omega_3^1(X)\,\,\, \mbox{and}\,\,\, \overline\omega_2^1(X)=-\overline\omega_3^4(X)\,.
\end{equation*}
If $M$ is minimal non-horizontal surface, then
$
-\langle A_{\xi}(f_4),f_4\rangle+\overline T ^5(\xi,f_5)=0,\, \forall \xi\in TM^\perp
$.
So,
\begin{equation*}
-\overline\omega^4_1(f_4)=A_5^3,\, -\overline\omega^4_2(f_4)=0,\, \overline\omega^4_3(f_4)=-A_5^1\,.
\end{equation*}
Therefore,
\begin{equation*}
\overline\nabla_{f_4}f_4=-A_5^3f_1+A_5^1f_3\,.
\end{equation*}
If we set
$
f_4=a^1e_1+a^2e_2+a^3e_3+a^4e_4
$, where $a^i:\mathbb{H}^2\rightarrow\mathbb{R}$ are smooth functions, then
\begin{align*}
f_2&=-Jf_4=a^3e_1+a^4e_2-a^1e_3-a^2e_4,\\
f_3&=Rf_4=-a^2e_1+a^1e_2+a^4e_3-a^3e_4,\\
f_1&=-Rf_2=a^4e_1-a^3e_2+a^2e_3-a^1e_4\,.
\end{align*}
Let $\varphi(t,s)$ be a parametrization of $M$ such that $\varphi_s(t,s)=f_4(\varphi(t,s))$. Then,
\begin{eqnarray}\label{system}
\frac{\dif}{\dif s}\left[\begin{array}{c}
a^1\\
a^2\\
a^3\\
a^4
\end{array}\right]=\left[
\begin{array}{cccc}
0&-\overline\omega^3_4(f_4)&0&\overline\omega^1_4(f_4)\\
\overline\omega^3_4(f_4)&0&-\overline\omega^1_4(f_4)&0\\
0&\overline\omega^1_4(f_4)&0&\overline\omega^3_4(f_4)\\
-\overline\omega^1_4(f_4)&0&-\overline\omega^3_4(f_4)&0
\end{array}\right]
\left[\begin{array}{c}
a^1\\
a^2\\
a^3\\
a^4
\end{array}\right]\,.
\end{eqnarray}
We shall denote by $O(4,\mathbb{R})$ the set of orthogonal matrices of order $4$. So,
if $\Phi(t,s)\in O(4,\mathbb{R})$ is a fundamental solution of (\ref{system}), then  it satisfies
\begin{eqnarray}\label{systemsol}
\left[\begin{array}{c}
a^1(t,s)\\
a^2(t,s)\\
a^3(t,s)\\
a^4(t,s)
\end{array}\right]=\Phi(t,s)
\left[\begin{array}{c}
a^1(t,0)\\
a^2(t,0)\\
a^3(t,0)\\
a^4(t,0)
\end{array}\right]\,.
\end{eqnarray}
Therefore, if we write $\varphi:\mathbb R^2\rightarrow \mathbb R^5$ as
\begin{equation*}
\varphi(t,s)=(x^1(t,s),x^2(t,s),x^3(t,s),x^4(t,s),x^5(t,s)),
\end{equation*}
then
$\dfrac{\partial x^i}{\partial s}=a^i$ for $i=1,2,3,4$ and $\dfrac{\partial x^5}{\partial s}=\dfrac 1 2 (a^3x^1+a^4x^2-a^1x^3-a^2x^4)$. From this we get
$$x^i(t,s)=x^i_t(0)+\int_0^sa^i_t(u)\dif u$$
for $i=1,2,3,4$ and
\begin{equation*}
x^5_t(s)=x^5_t(0)+\frac 1 2\int_0^s\Big(a^3_{t}(u)x^1_t(u)+a^4_{t}(u)x^2_t(u)-a^1_{t}(u)x^3_t(u)-a^2_{t}(u)x^4_t(u)\Big)\dif u,
\end{equation*}
where $a^i_{t}(s)=a^i(t,s)$ and $x^{i}_t(s)=x^i(t,s)$.

In addition, if we impose the following condition
$
\left\langle\dfrac{\partial \varphi}{\partial t}(t,0),\dfrac{\partial \varphi}{\partial s}(t,0)\right\rangle=0
$
then $\dfrac{\partial \varphi}{\partial t}(t,0)=e^5\left(\dfrac{\partial \varphi}{\partial t}(t,0)\right)f_5(t,0)$ and
\begin{equation*}
\varphi(t,0)=\!\!\int_0^t\!\!r(v)\Big(e_5(v,0)-A_5^1(v,0)f_1(v,0)-A_5^2(v,0)f_2(v,0)-A_5^3(v,0)f_3(v,0)\Big)\dif v,
\end{equation*}
where $r(t)=e^5\Big(\dfrac{\partial \varphi}{\partial t}(t,0)\Big)$.

Using all of the above we can study two models of minimal non-horizontal surfaces in $\mathbb{H}^2$.
\begin{ex} Ruled surface
\end{ex}
 Consider the case where $\overline\omega_4^1(f_4)=\overline\omega_4^3(f_4)=0$. In this case $a^i,i=1,\cdots,4$ are constant on the tangent curves on $TM\cap D$. So,
\begin{equation}\label{x1234}
x^i(t,s)=x^i(t)+a^i(t)s, i=1,2,3,4
\end{equation}
where $x^i(t)=x^i(t,0)$ and $a^i(t)=a^i(t,0)$. Also, we get
\begin{equation}\label{x5}
x^5(t,s)=x^5(t)+\frac{s}{2}\Big(a^3(t)x^1(t)+a^4(t)x^2(t)-a^1(t)x^3(t)-a^2(t)x^4(t)\Big).
\end{equation}
Moreover,
\begin{equation*}
\varphi(t,0)=\int_0^tr(v)\Big(e_5(v,0)-A_5^2(v,0)f_2(v,0)\Big)\dif v\,.
\end{equation*}
Now, we write
\begin{align*}
\varphi(t,s)-\varphi(t,0)=&\sum_{i=1}^5\Big(x^i(t,s)-x^i(t,0)\Big)\frac{\partial}{\partial x^i}\\
=&\sum_{i=1}^4\Big(x^i(t,s)-x^i(t,0)\Big)e_i(t)+\Big(x^5(t,s)-x^5(t,0)\\
&+\frac 1 2\Big((x^1(t,s)-x^1(t,0))x^3(t,0)+(x^2(t,s)-x^2(t,0))x^4(t,0)\\
&-(x^3(t,s)-x^3(t,0))x^1(t,0)-(x^4(t,s)-x^4(t,0))x^2(t,0)\Big)\Big)e_5(t)\,.
\end{align*}
It follows from (\ref{x1234}) and (\ref{x5}) that
\begin{align*}
\varphi(t,s)=\varphi(t,0)+s f_4(t,0)\,.
\end{align*}

\begin{ex} Tubular surfaces
\end{ex}
Consider $\overline\omega_4^1(f_4)=b_1(t)$ and $\overline\omega_4^3(f_4)=b_3(t)$ constants along tangent curves on $TM\cap D$ such that $b_1(t)^2+b_3(t)^2\neq0$.
The eigenvalues of the matrix (\ref{system}) are $\pm i\sqrt{b_1(t)^2+b_3(t)^2}$, where $i=\sqrt{-1}$. Resolving the differential equation (\ref{system}) we obtain the fundamental matrix
\begin{equation*}
\Phi(t,s)=\left[
\begin{array}{cccc}
\cos(sb(t))&-\dfrac{b_3(t)\sin(sb(t))}{b(t)}&0&\dfrac{b_1(t)\sin(sb(t))}{b(t)}\\
\dfrac{b_3(t)\sin(sb(t))}{b(t)}&\cos(sb(t))&-\dfrac{b_1(t)\sin(sb(t))}{b(t)}&0\\
0&\dfrac{b_1(t)\sin(sb(t))}{b(t)}&\cos(sb(t))&\dfrac{b_3(t)\sin(sb(t))}{b(t)}\\
-\dfrac{b_1(t)\sin(sb(t))}{b(t)}&0&-\dfrac{b_3(t)\sin(sb(t))}{b(t)}&\cos(sb(t))
\end{array}\right],
\end{equation*} where $b(t)=\sqrt{b_1(t)^2+b_3(t)^2}$. Then
\begin{align*}
a^1(t,s)&=\cos(sb(t))a^1(t)-\frac{b_3(t)\sin(sb(t))}{b(t)}a^2(t)+\frac{b_1(t)\sin(sb(t))}{b(t)}a^4(t),\\
a^2(t,s)&=\frac{b_3(t)\sin(sb(t))}{b(t)}a^1(t)+\cos(sb(t))a^2(t)-\frac{b_1(t)\sin(sb(t))}{b(t)}a^3(t),\\
a^3(t,s)&=\frac{b_1(t)\sin(sb(t))}{b(t)}a^2(t)+\cos(sb(t))a^3(t)+\frac{b_3(t)\sin(sb(t))}{b(t)}a^4(t),\\
a^4(t,s)&=-\frac{b_1(t)\sin(sb(t))}{b(t)}a^1(t)-\frac{b_3(t)\sin(sb(t))}{b(t)}a^3(t)+\cos(sb(t))a^4(t)\,.
\end{align*}
It follows that
{\small{\begin{align*}
x^1(t,s)&=x^1(t)+\frac{\sin(sb(t))}{b(t)}a^1(t)+b_3(t)\frac{\cos(sb(t))-1}{b(t)^2}a^2(t)-b_1(t)\frac{\cos(sb(t))-1}{b(t)^2}a^4(t),\\
x^2(t,s)&=x^2(t)-b_3(t)\frac{\cos(sb(t))-1}{b(t)^2}a^1(t)+\frac{\sin(sb(t))}{b(t)}a^2(t)+b_1(t)\frac{\cos(sb(t))-1}{b(t)^2}a^3(t),\\
x^3(t,s)&=x^3(t)-b_1(t)\frac{\cos(sb(t))-1}{b(t)^2}a^2(t)+\frac{\sin(sb(t))}{b(t)}a^3(t)-b_3(t)\frac{\cos(sb(t))-1}{b(t)^2}a^4(t),\\
x^4(t,s)&=x^4(t)+b_1(t)\frac{\cos(sb(t))-1}{b(t)^2}a^1(t)+b_3(t)\frac{\cos(sb(t))-1}{b(t)^2}a^3(t)+\frac{\sin(sb(t))}{b(t)}a^4(t)\,.
\end{align*}}}
Plugging these terms in (\ref{x5}) we get
\begin{align*}
x^5(t,s)=&\,x^5(t)+\frac{\cos(sb(t))-1}{2b(t)^2}\Big(-(b_1(t)a^2(t)+b_3(t)a^4(t))x^1(t)\\
&+(b_1(t)a^1(t)+b_3(t)a^3(t))x^2(t)-(b_3(t)a^2(t)-b_1(t)a^4(t))x^3(t)\\
&-(b_3(t)a^1(t)+b_1(t)a^3(t))x^4(t)\Big)+\frac{\sin(sb(t))}{2b(t)}\Big(a^3(t)x^1(t)+a^4(t)x^2(t)\\
&-a^1(t)x^3(t)-a^2(t)x^4(t)\Big).
\end{align*}
Then, the parametrization of a minimal non-horizontal surface in $\mathbb{H}^2$ is
{\small{\begin{align*}
\varphi(t,s)-\varphi(t,0)=&\,\frac{\sin(sb(t))}{b(t)}\sum_{i=1}^4a^i(t)e_i(t)\\
&+\frac{\cos(sb(t))-1}{b(t)^2}b_3(t)\Big(a^2(t)e_1(t)-a^1(t)e_2(t)-a^4(t)e_3(t)\\
&+a^3(t)e_4(t)\Big)
+\frac{\cos(sb(t))-1}{b(t)^2}b_1(t)\Big(-a^4(t)e_1(t)+a^3(t)e_2(t)\\
&-a^2(t)e_3(t)+a^3(t)e_4(t)\Big)+\Big(x^5(t,s)-x^5(t,0)\\
&+\frac 1 2(x^1(t,s)-x^1(t,0))x^3(t,0)+\frac 1 2(x^2(t,s)-x^2(t,0))x^4(t,0)\\
&-\frac 1 2(x^3(t,s)-x^3(t,0))x^1(t,0)-\frac 1 2(x^4(t,s)-x^4(t,0))x^2(t,0)\Big)e_5(t)\, ,
\end{align*}}}
Hence,
{\small{\begin{equation}\label{tubular}
\varphi(t,s)=\varphi(t,0)+\frac{\sin(sb(t))}{b(t)}f_4(t,0)-\frac{\cos(sb(t))-1}{b(t)^2}\Big(b_3(t)f_3(t,0)+b_1(t)f_1(t,0)\Big).
\end{equation}}}

Conversely, suppose that $\gamma$ is a transverse curve in $\mathbb{H}^2$ ($\gamma$ is a one dimensional non-horizontal submanifod) and $f_4\in D$ is a unitary vector field along  $\gamma$. According to Corollary \ref{baseD} we can choose $f_2=-Jf_4,f_3=Rf_4, f_1=-Rf_2$ which is an orthonormal basis of $D$ along $\gamma$. We complete this basis to a basis of $T\mathbb H^2$ with $f_5=e_5-\sum^3_{\alpha=1}A_5^\alpha f_\alpha$ such that $\dfrac{\dif\gamma}{\dif t}=\lambda_1 f_4+\lambda_2 f_5$, where $\lambda_1, \lambda_2$ are non-null smooth functions on $\mathbb{H}^2$.

Now, if we define a non-horizontal surface $M$ by (\ref{tubular}) with
$\varphi(t,0)=\gamma(t,0)$, $b_1(t)=-A_5^3(t,0)$ and $b_3(t)=A_5^1(t,0)$, it follows that $M$ is minimal.
\begin{ex}{\emph{
Let $\gamma(t)=\left(r\cos\Big(\frac{t}{r}\Big),0,r\sin\Big(\frac{t}{r}\Big),0,0\right)$ be a transverse curve in $\mathbb{H}^2$ and let $ f_4(t)=\left(0,\cos\Big(\frac{t}{r}\Big),0,\sin\Big(\frac{t}{r}\Big),0\right)$
be an unitary vector field along $\gamma$. Then,
\begin{align*}
f_1(t)&=\sin\Big(\frac{t}{r}\Big)e_1(t)+\cos\Big(\frac{t}{r}\Big)e_3(t),\\
f_2(t)&=\sin\Big(\frac{t}{r}\Big)e_2(t)-\cos\Big(\frac{t}{r}\Big)e_4(t),\\
f_3(t)&=-\cos\Big(\frac{t}{r}\Big)e_1(t)+\sin\Big(\frac{t}{r}\Big)e_3(t)\,.
\end{align*}
Moreover,
\begin{equation*}
\frac{\dif\gamma}{\dif t}=\left(-\sin\Big(\frac{t}{r}\Big),0,\cos\Big(\frac{t}{r}\Big),0,0\right)=\frac{r}{2}\left(\frac{2}{r}\cos\Big(\frac{2t}{r}\Big)f_1(t)
+\frac{2}{r}\sin\Big(\frac{2t}{r}\Big)f_3(t)+e_5(t)\right)
\end{equation*}
and
\begin{equation*}
f_5(t)=e_5(t)+\frac{2}{r}\cos\Big(\frac{2t}{r}\Big)f_1(t)+\frac{2}{r}\sin\Big(\frac{2t}{r}\Big)f_3(t).
\end{equation*}
So,
\begin{equation*}
b_1(t)=A_5^3(t,0)=-\frac{2}{r}\sin\Big(\frac{2t}{r}\Big),\quad \mbox{and}\quad b_3(t)=-A_5^1(t,0)=\frac{2}{r}\cos\Big(\frac{2t}{r}\Big)
\end{equation*}
 and we obtain that
 $
 b(t)=\sqrt{b_1(t)^2+b_3(t)^2}=\dfrac{2}{r}.
 $
Therefore, the parametrization $\varphi(t,s)$ of minimal non-horizontal surface in $\mathbb{H}^2$ generated by $\gamma(t)$ and the unitary vector field $f_4(t)$ is
\begin{align*}
\varphi(t,s)=&\frac{r}{2}\Big(\cos\Big(\frac{t}{r}\Big)\Big(1+\cos\Big(\frac{2s}{r}\Big)\Big),\sin\Big(\frac{2s}{r}\Big)\cos\Big(\frac{t}{r}\Big),\\
&\sin\Big(\frac{t}{r}\Big)\Big(1+\cos\Big(\frac{2s}{r}\Big)\Big),\sin\Big(\frac{2s}{r}\Big)\sin\Big(\frac{t}{r}\Big),0\Big).
\end{align*}
This surface is non-horizontal in all points except $(0,0,0,0,0)$.}}
\end{ex}
\section{Minimal hypersufaces}
Let $
M=\{x\in \mathbb G :\phi(x)=0\,\, \mbox{e}\,\, \mbox{grad}_{D}\phi(x)\neq 0\}
$
 be a non-horizontal hypersurface of $\mathbb G$,
where $\phi:\mathbb G\rightarrow\mathbb R$ is a real function $C^\infty$ and $\mbox{grad}_{D}$ denotes the horizontal gradient operator defined by
\begin{equation*}
\mbox{grad}_D\phi=\sum_{i=1}^{d_1}\phi_ie_i\,\,\, \mbox{or}\,\,\, \dif\phi(X)=\langle \mbox{grad}_{D}\phi,X\rangle,\, \forall\, X\in D,
 \end{equation*}
 where $\phi_i=e_i(\phi)$.

To introduce the notion of divergence,  observe that the volume form \eqref{volumeG} is parallel with respect to $\overline\nabla$. According to \cite[p.~282]{Kobayashi1963}, for every vector field $X$ on $\mathbb{G}$ we have
\begin{equation*}
\mbox{div}_{\mathbb G}X=\sum_{k=1}^{n}e^k(\overline\nabla_{e_k}X).
\end{equation*}
Let
$f_1=\frac{\mbox{grad}_D\phi}{N}
$,
where
$
N=\sqrt{\sum_{i=1}^{d_1}\phi_i^2},
$
and let $f_2,\ldots,f_n$ be the adapted basis on $TM$.

\begin{prop}\label{divHiper}
$$-H_{f_1}=\emph{div}_{\mathbb G}\left(\frac{\emph{grad}_{D}\phi}{|\emph{grad}_{D}\phi|}\right)\,.$$
\end{prop}
\begin{proof}  From the definition of the mean curvature it follows that
\begin{align*}
-H_{f_1}=&-\sum_{i=2}^{d_1}\langle A_{f_1}(f_i),f_i\rangle=\sum_{i=2}^{d_1}\langle \overline\nabla_{f_i}f_1,f_i\rangle=\sum_{i=2}^{d_1}\sum_{j=1}^{d_1}\left\langle \overline\nabla_{f_i}\Big(\frac{\phi_j}{N}e_j\Big),f_i\right\rangle\\
=&\sum_{j=1}^{d_1}\sum_{i=2}^{d_1}\langle e_j,f_i\rangle {f_i}\Big(\frac{\phi_j}{N}\Big)=\sum_{j=1}^{d_1}\Big(e_j-\langle e_j,f_1 \rangle f_1\Big)\Big(\frac{\phi_j}{N}\Big)\\
=&\sum_{j=1}^{d_1}e_j\Big(\frac{\phi_j}{N}\Big)-\sum_{j=1}^{d_1}\langle e_j,f_1 \rangle f_1\Big(\langle e_j,f_1 \rangle\Big)\\
=&\,\mbox{div}_{\mathbb G}\Big(\frac{\mbox{grad}_{D}\phi}{|\mbox{grad}_{D}\phi|}\Big)-\frac 1 2 f_1\Big(\sum_{j=1}^{d_1}\langle e_j,f_1 \rangle^2\Big)=\mbox{div}_{\mathbb G}\left(\frac{\mbox{grad}_{D}\phi}{|\mbox{grad}_{D}\phi|}\right),
\end{align*}
since $\sum_{j=1}^{d_1}\langle e_j,f_1 \rangle^2=1$.
\end{proof}
\begin{cor}
Let $M\subset\mathbb{H}^n$ be a  non-horizontal hypersurface defined by the function $\phi=u(x_1,\cdots,x_{2n})-x_{2n+1}=0$. Then, $M$ is minimal if and only if
\begin{equation}\label{Apgrafico}
\sum_{i=1}^{2n}u_{i,i} -\frac{1}{N^2}\sum_{i,j=1}^{2n}\phi_i\phi_ju_{i,j}=0,
\end{equation}
where $u_{i,j}=\frac{\partial u}{\partial x_i\partial x_j}, \phi_j=e_j(\phi)$ and $N=\sqrt{\phi_1^2+\cdots+\phi_{2n}^2}$.
\end{cor}
\begin{proof} For $j=1,\cdots,n, \phi_j=u_{x_j}-\frac 1 2 x_{j+n},\, \phi_{j+n}=u_{x_{j+n}}+\frac 1 2 x_j$,  where $u_{x_{k}}=\frac{\partial u}{\partial x_k}$. So
$e_i(\phi_j)=u_{x_i,x_j},\, e_i(\phi_{j+n})=u_{x_i,x_{j+n}}+\frac 1 2\delta_{ij},\,
e_{i+n}(\phi_j)=u_{x_{i+n},x_j}-\frac 1 2 \delta_{ij}$ and $e_{i+n}(\phi_{j+n})=u_{x_{i+n},x_{j+n}}$, for $i,j=1,\cdots,n$.
Therefore,
\begin{align*}
-H_{f_1}=&\sum_{i=1}^{2n}e_i\frac{\phi_i}{N}=\sum_{i=1}^{2n}\left(\frac{e_i(\phi_i)}{N}-\phi_i\sum_{j=1}^{2n}\frac{\phi_je_i(\phi_j)}{N^3}\right)\\
=&\frac{1}{N}\left(\sum_{i=1}^{2n}u_{i,i} -\frac{1}{N^2}\sum_{i,j=1}^{2n}\phi_i\phi_ju_{i,j} \right)\,.
\end{align*}
   \end{proof}
From this proposition it follows immediately that the \emph{hyperbolic paraboloid} with
\begin{equation*}
u(x_1,\ldots,x_{2n})=\frac 1 4\sum_{i=1}^n(x_i^2-x_{i+n}^2)
\end{equation*} is minimal as in \cite{Montefalcone2012}.


\end{document}